\documentclass{amsart}

\usepackage{amsmath}
\usepackage{amsxtra}
\usepackage{amscd}
\usepackage{amsthm}
\usepackage{amsfonts}
\usepackage{amssymb}
\usepackage{color}
\usepackage{eucal}
\usepackage[all]{xy}
\usepackage{graphicx}
\usepackage{mathabx}
\usepackage{hyperref}

\newtheorem{theorem}{Theorem}[section]
\newtheorem*{theorem*}{Theorem}
\newtheorem{corollary}[theorem]{Corollary}

\newtheorem{lemma}[theorem]{Lemma}
\newtheorem{proposition}[theorem]{Proposition}
\newtheorem*{proposition*}{Proposition}

\newtheorem*{conjecture*}{Conjecture}
\newtheorem{expectation}[theorem]{Expectation}
\newtheorem{thm-dfn}[theorem]{Theorem-Definition}
\newtheorem{clm-dfn}[theorem]{Claim-Definition}
\newtheorem{claim}[theorem]{Claim}

\theoremstyle{definition}
\newtheorem{definition}[theorem]{Definition}
\newtheorem{notation}[theorem]{Notation}
\newtheorem{remark}[theorem]{Remark}
\newtheorem*{remark*}{Remark}

\numberwithin{equation}{section}

\newcommand{\quash}[1]{}  

\newcommand{\fraka}{{\mathfrak a}}
\newcommand{\frakb}{{\mathfrak b}}

\newcommand{\frakg}{{\mathfrak g}}
\newcommand{\frakh}{{\mathfrak h}}

\newcommand{\frakk}{{\mathfrak k}}

\newcommand{\frakn}{{\mathfrak n}}

\newcommand{\frakz}{{\mathfrak z}}

\newcommand{\bbC}{{\mathbb C}}

\newcommand{\bbR}{{\mathbb R}}

\newcommand{\bbZ}{{\mathbb Z}}

\newcommand{\calB}{{\mathcal B}}
\newcommand{\calC}{{\mathcal C}}

\newcommand{\calJ}{{\mathcal J}}

\newcommand{\calM}{{\mathcal M}}

\newcommand{\calO}{{\mathcal O}}

\newcommand{\calV}{{\mathcal V}}

\newcommand{\calZ}{{\mathcal Z}}

\newcommand{\temp}{\textnormal{temp}}
\newcommand{\Hom}{\textnormal{Hom}}

\begin{document}

\title[Second adjointness for tempered admissible representations]{Second adjointness for tempered admissible representations of a real group}
        \author{Alexander Yom Din}
        \address{Department of Mathematics, California Institute of Technology, Pasadena, CA 91125, USA.}
         \email{ayomdin@gmail.com}

\maketitle

\begin{abstract}
	We study second adjointness in the context of tempered admissible representations of a real reductive group. Compared to a recent result of Crisp and Higson, this generalizes from $SL_2$ to a general group, but specializes to only considering admissible representations. We also discuss Casselman's canonical pairing in this context, and the relation to Bernstein morphisms. Additionally, we take the opportunity to discuss some relevant functors and some of their relations.
\end{abstract}

\tableofcontents


\section{Introduction}

\subsection{Second adjointness}

Let $G$ be a connected reductive group over a local field $F$. Let $P , P^- \subset G$ be opposite parabolics defined over $F$, with Levi $L = P \cap P^-$. One has the \emph{functors of parabolic restriction and induction} w.r.t. $P$, which form an adjunction $$ pres : \calM (G(F)) \rightleftarrows \calM (L(F)) : pind$$ (meaning that $pres$ is the left adjoint of $pind$). Here $\calM (\cdot)$ is the category of smooth representations (over $\bbC$) in the case when $F$ is non-archimedean, and is the category of smooth Frechet representations of moderate growth (over $\bbC$) in the case when $F$ is archimedean. The functor $pres$ is usually also known as the \emph{Jacquet functor}. Let us denote similarly by $$ pres^- : \calM (G(F)) \to \calM (L(F))$$ the parabolic restriction where we use the parabolic $P^-$ instead of $P$.

\medskip

The functor $pind$ is exact. In the non-archimedean case, the functor $pres$ is exact as well (this is a basic result of Jacquet) and one has the fundamental \emph{second adjointness} theorem of Joseph Bernstein:

\begin{theorem*}[J. Bernstein]
	Suppose that $F$ is non-archimedean. Then there is a canonical adjunction $$ pind : \calM (L(F)) \rightleftarrows \calM (G(F)) : pres^-.$$
\end{theorem*}

In the archimedean case, things become more complicated - the functor $pres$ is not exact and second adjointness does not hold in its above formulation.

\medskip

\begin{center}\textbf{Let us from now on assume that $F = \bbR$.}\end{center}

\medskip

Let us consider the subcategories $\calM (\cdot)_{\temp} \subset \calM (\cdot)$ of \emph{tempered} representations (those are, morally, representations whose matrix coefficients are close to being square integrable, and thus who have a chance of contributing to the Plancherel decomposition of $L^2 (G(\bbR))$\footnote{Actually, modulo the center, as usual.}). The functor $pind$ preserves these, but $pres$ does not. Nevertheless, one still has an adjunction $$ temppres : \calM (G(\bbR))_{\temp} \rightleftarrows \calM (L(\bbR))_{\temp} : pind,$$ where $temppres (V)$ is the biggest tempered quotient of $pres (V)$. Of course, we also denote by $temppres^-$ the analogous functor where we use $P^-$ instead of $P$.

\medskip

It was relatively recently shown by T. Crisp and N. Higson:

\begin{theorem*}[\cite{CrHi}]
	Suppose that $G = SL_2$. Then there is a canonical adjunction $$ pind : \calM (L(\bbR))_{\temp} \rightleftarrows \calM (G(\bbR))_{\temp} : temppres^-.$$
\end{theorem*}

Let us consider the subcategories $\calM^a (\cdot) \subset \calM (\cdot)$ of admissible representations (we use terminology where those are the representations whose underlying $(\frakg , K)$-module is of finite length). The main observation of this paper is that $$ temppres : \calM^a (G(\bbR))_{\temp} \to \calM^a (L(\bbR))_{\temp}$$ is exact (Proposition \ref{temppres exact}), and the following theorem holds:

\begin{theorem*}[Theorem \ref{thm second adj temp}]
	There is a canonical adjunction $$ pind : \calM^a (L(\bbR))_{\temp} \rightleftarrows \calM^a (G(\bbR))_{\temp} : temppres^-.$$
\end{theorem*}

\begin{remark*}
	Thus, relative to the result of \cite{CrHi}, we generalize from $SL_2$ to a general group, but specialize to only considering admissible representations. In this paper, we don't deal with non-admissible representations.
\end{remark*}

\subsection{Canonical pairing}

In the admissible case, second adjointness is easily shown to be equivalent to the existence and non-degeneracy of \emph{Casselman's canonical pairing} between Jacquet modules. In our setting, this is the following. Denote by $$ (\cdot)^{\vee} : \calM^a (G(\bbR)) \xrightarrow{\approx} \calM^a (G(\bbR))^{\textnormal{op}}$$ the functor of passing to the contragradient representation. Then Theorem \ref{thm second adj temp} above is equivalent to:

\begin{theorem*}[Theorem \ref{thm Cass pair tempered}]
	Let $V\in \calM^a (G(\bbR))_{\temp}$. Then there is a canonical non-degenerate pairing $$ temppres^- (V) \otimes temppres (V^{\vee}) \to \bbC.$$
\end{theorem*}

The point of restricting attention to tempered representations in the archimedean case, from a technical perspective, is as follows. In the archimedean case, when one considers not necessarily tempered representations, Casselman's canonical pairing exists between Casselman-Jacquet modules rather than Jacquet modules (in contrast with the non-archimedean case). The non-exactness of the Jacquet functor is responsible for this pairing not passing to a pairing between Jacquet modules. However, when one restricts attention to tempered representations, the possible exponents have a conical constraint, which causes the reading of $temppres$ from the Casselman-Jacquet module to be exact, and things are again orderly.

\subsection{Relation to Bernstein morphisms}

In \cite{DeKnKrSc}, the authors construct Bernstein morphisms for real spherical varieties, as \cite{SaVe} did for non-archimedean spherical varieties, both following ideas of J. Bernstein. In a special case of the general setting, relevant for the current paper, this is an isometric embedding $$ \textnormal{Ber}_I : L^2 \left( (G(\bbR) \times G(\bbR)) / (\Delta L(\bbR) \cdot (N^- (\bbR) \times N (\bbR))) \right) \to L^2 \left( G(\bbR) \right)$$ (where $N , N^-$ are the unipotent radicals of $P,P^-$).

\medskip

In \S\ref{sec Bernstein morphisms} we will indicate how the canonical pairing for tempered admissible representations of this paper should be related to the construction of $\textnormal{Ber}_I$. The verification should be a straight-forward translation between the languages of \cite{DeKnKrSc} and the current paper, but we don't try to present details here.

\subsection{Non-tempered admissible representations}

The purpose of the second part of this paper is twofold. First, in section \S\ref{sec functors}, we would like to record some of the ideas from our Ph.D. thesis \cite{Yo1} in a bit more organized and complete way. Second, in section \S\ref{sec second adj second take}, we will use this to present the proof of Theorem \ref{thm second adj temp} in a different way, which gives another point of view, putting an emphasis on what is the right adjoint of $pind$ when one considers not necessarily tempered representations, and why it differs from $pres^-$.

\medskip

Namely, it is explained that the right adjoint of $$pind : \calM^a (L(\bbR)) \to \calM^a (G(\bbR))$$ is $$ V \mapsto \bbC_{\rho_P} \otimes \calJ_P (V)^{\frakn},$$ while the functor $pres^-$ is given by $$ V \mapsto \bbC_{\rho_P} \otimes \calJ_P (V) / \frakn^- \calJ_P (V),$$ and the former functor has an obvious map into the latter. Here $\calJ_P (V)$ is the \emph{Casselman-Jacquet} module, $\frakn , \frakn^-$ are the Lie algebras of $N , N^-$, and $\bbC_{\rho_P} \otimes -$ are some standard $\rho$-twists.

\medskip

We plan to further study this situation for non-tempered representations in the future.

\subsection{Dissatisfaction}

Throughout the paper, we use some analytical inputs, the main one being Casselman's canonical pairing. It is our hope that in the future we will be able to treat all of these inputs algebraically.

\subsection{Acknowledgments}

We would like to thank D. Kazhdan for suggesting us to prove the result of \cite{CrHi} for a general group, using our techniques. We would like to thank J. Bernstein, E. Sayag and H. Schilchtkrull for useful conversations. We would like to thank Y. Sakellaridis for useful correspondence. We would like to thank an anonymous referee for very useful exposition remarks.

\section{Setting and notations}\label{sec set and not}

\subsection{The group}

We fix the following. Let $G$ be a connected reductive algebraic group over $\bbC$, together with a real form $\sigma$ (so $G(\bbR) = G^{\sigma}$). Let $\theta$ be a Cartan involution of $(G , \sigma)$. Let $K := G^{\theta}$ be the resulting complexification of the maximal compact subgroup $K(\bbR) = G(\bbR)^{\theta}$. We denote by $\frakg , \frakg_{\bbR}$ the Lie algebras of $G, G(\bbR)$. We choose a maximal abelian subspace $\fraka \subset \frakg_{\bbR}^{\theta , -1}$. We denote by $R\subset \fraka^*$ the subset of roots. We choose a system of positive roots $R^+ \subset R$, with simple roots $\Sigma \subset R^+$. For $I \subset \Sigma$, we have the corresponding standard parabolic $G_I \cdot N_{(I)} \subset G$ (where $G_I$ is the Levi subgroup and $N_{(I)}$ is the unipotent radical), and also its opposite $G_I \cdot N_{(I)}^- \subset G$. We set $K_I := K \cap G_I = K \cap (G_I \cdot N_{(I)})$. For example, $G_{\Sigma} = G$. We use the standard Gothic notations for corresponding Lie algebras.

\medskip

Let $I \subset \Sigma$. We denote $$ R^+_I := R^+ \cap \left( \sum_{\alpha \in I} \bbZ_{\ge 0} \cdot \alpha \right), \quad R^+_{(I)} := R^+ \setminus R^+_I.$$ We denote $$ \rho_{I} = \frac{1}{2} \sum_{\alpha \in R^+_I} \alpha \in \fraka^* , \quad \rho_{(I)} = \frac{1}{2} \sum_{\alpha \in R^+_{(I)}} \alpha \in \fraka^* .$$ We denote $$\fraka_{cent , I} := \frakz_{\fraka} (\frakg_I) = \{ H \in \fraka \ | \ \alpha (H) = 0 \ \forall \alpha \in I\}  \subset \fraka.$$ Also, we denote $$ \fraka^{+,I} := \{ H \in \fraka \ | \ \alpha (H) \ge 0 \ \forall \alpha \in I\} \subset \fraka.$$ Finally, we denote by $\leq_I$ the partial order on $\fraka^*$ given by $\lambda \leq_I \mu$ if $(\mu - \lambda)(H) \ge 0$ for all $H \in \fraka^{+,I}$.

\subsection{Modules}

Let $\frakh$ be a reductive Lie algebra. We denote by $\calM (\frakh)$ the abelian category of $\frakh$-modules. By an \emph{admissible} $\frakh$-module, we understand an $\frakh$-module $V$ which is finitely generated over $U(\frakh)$ and is $Z(\frakh)$-finite. We denote by $$ \calM^a (\frakh) \subset \calM (\frakh)$$ the full subcategory of admissible modules.

\medskip

For an Harish-Chandra pair $(\frakh , L)$, we denote by $\calM (\frakh , L)$ the abelian category of $(\frakh ,L)$-modules. We say that an $(\frakh , L)$-module is admissible if it is admissible as an $\frakh$-module, and denote by $$ \calM^a (\frakh , L) \subset \calM (\frakh , L)$$ the full subcategory of admissible modules.

\medskip

For a complex reductive group $L$, we denote by $\hat{L}$ the set of isomorphism classes of irreducible algebraic representations of $L$. Given an algebraic representation $V$ of $L$, and $\alpha \in \hat{L}$, we denote by $V^{[\alpha]} \subset V$ the $\alpha$-isotypic subspace.

\medskip

Given a commutative real Lie algebra $\frakb$ and a locally-finite complex $\frakb$-module $V$, we denote by $wt_{\frakb} (V) \subset \frakb^*_{\bbC}$ the set of generalized eigenweights of $\frakb$ on $V$, and for $\lambda \in \frakb^*_{\bbC}$ we denote by $V_{\frakb , \lambda}$ the subspace of $V$ consisting of vectors with generalized eigenweight $\lambda$ with respect to $\frakb$. The following are very useful claims about admissibility.

\begin{lemma}\label{clm admissible 1}
	For $V \in \calM (\frakg , K)$, the following are equivalent:
	
	\begin{enumerate}
		\item $V$ is admissible.
		\item $V$ is finitely generated over $U(\frakg)$ and $V^{[\alpha]}$ are finite-dimensional for all $\alpha \in \hat{K}$.
		\item $V$ is $Z(\frakg)$-finite and $V^{[\alpha]}$ are finite-dimensional for all $\alpha \in \hat{K}$.
		\item $V$ has finite length.
	\end{enumerate}
\end{lemma}

\begin{proof}
	Given in Appendix \ref{sec proofs}.
\end{proof}

\begin{lemma}\label{clm admissible 2}
	For $V \in \calM (\frakg , K_I N_{(I)})$, the following are equivalent:
	
	\begin{enumerate}
		\item $V$ is admissible.
		\item $V$ is $Z(\frakg)$-finite and\footnote{Here, and elsewhere, we use the notation $W^{\frakn^k} := \{ w \in W \ |  \frakn^k w = 0\}$.} $V^{\frakn_{(I)}}$ is an admissible $(\frakg_I , K_I)$-module.
		\item $V$ is $Z(\frakg)$-finite and $V^{\frakn_{(I)}^k}$ are admissible $(\frakg_I , K_I)$-modules for every $k \in \bbZ_{\ge 1}$.
		\item $V$ is $Z(\frakg)$-finite, $\fraka_{cent,I}$-locally finite, and every generalized eigenspace $V_{\fraka_{cent,I} , \lambda}$ is an admissible $(\frakg_I , K_I)$-module.
		\item $V$ has finite length.
	\end{enumerate}
\end{lemma}

\begin{proof}
	Given in Appendix \ref{sec proofs}.
\end{proof}

Recall also that the forgetful functor $\calM (\frakg , K_I N_{(I)}) \to \calM (\frakg , K_I)$ is fully faithful, and the essential image consists of $(\frakg , K_I)$-modules which are locally $\frakn_{(I)}$-torsion. We will therefore think of $\calM (\frakg , K_I N_{(I)})$ as a full subcategory of $\calM (\frakg , K_I)$ in what follows.

\begin{lemma}\label{clm admissible 3}
	For $V \in \calM (\frakg , K_I)$, the following are equivalent:

	\begin{enumerate}
		\item $V$ is a $(\frakg , K_I N_{(I)})$-module, and admissible as such.
		\item $V$ is $Z(\frakg)$-finite, $\fraka_{cent,I}$-locally finite, and every generalized eigenspace $V_{\fraka_{cent,I} , \lambda}$ is an admissible $(\frakg_I , K_I)$-module. In addition, there exists a finite set $S \subset (\fraka_{cent,I})^{*}_{\bbC}$ such that $$ wt_{\fraka_{cent,I}} (V) \subset S - \sum_{\alpha \in wt_{\fraka_{cent,I}} (\frakn_{(I)})} \bbZ_{\ge 0} \cdot \alpha.$$
	\end{enumerate}
\end{lemma}

\begin{proof}
	Given in Appendix \ref{sec proofs}.
\end{proof}

\subsection{Dualities}

Recall the contragradient duality $$  (\cdot)^{\vee} : \calM^a (\frakg , K_I N_{(I)}) \xrightarrow{\approx} \calM^a (\frakg , K_I N_{(I)}^-)^{op} $$ given by $$ V^{\vee} := (V^*)^{K_I\textnormal{-finite} , \ \frakn_{(I)}^- \textnormal{-torsion}}.$$ In particular, for $I = \Sigma$, we obtain the contragradient duality $$ (\cdot)^{\vee} : \calM^a (\frakg , K) \xrightarrow{\approx} \calM^a (\frakg , K)^{op}.$$

\begin{lemma}\label{lem dual}
	The formula given for $(\cdot)^{\vee}$ indeed defines a duality as stated. One also has the description: $$ (V^*)^{K_I\textnormal{-finite} , \ \frakn_{(I)}^- \textnormal{-torsion}} = (V^*)^{K_I\textnormal{-finite} , \ \fraka_{cent,I}\textnormal{-finite}}.$$
\end{lemma}

\begin{proof}
	This is well-known, and not hard to establish based on all the admissibility Lemmas of this paper, so left as an exercise.
\end{proof}

\subsection{Representations}

In this paper we prefer to work with $(\frakg , K)$-modules rather than with representations. Let us briefly recall the relation.

\medskip

We denote by $\calM (G(\bbR))$ the category of smooth Frechet representation of $G(\bbR)$ which are of moderate growth. We have the functor $$(\cdot)^{[K]} : \calM (G(\bbR)) \to \calM (\frakg , K)$$ of passing to $K(\bbR)$-finite vectors, and we say that a representation $\calV \in \calM (G(\bbR))$ is \emph{admissible} if $\calV^{[K]}$ is an admissible $(\frakg , K)$-module. We denote by $$ \calM^a (G(\bbR)) \subset \calM (G(\bbR))$$ the full subcategory of admissible representations.

\medskip

The following is the basic theorem:

\begin{theorem}[Casselman-Wallach, \cite{Ca}, \cite{Wa1}, {\cite[\S 11]{Wa3}}]\label{Casselman-Wallach}
	The functor $$ (\cdot)^{[K]} : \calM^a (G(\bbR)) \to \calM^a (\frakg , K)$$ is an equivalence of categories.
\end{theorem}

We will denote by $(\cdot)^{\infty}$ the equivalence of categories inverse to that in Theorem \ref{Casselman-Wallach}.

\section{Casselman's canonical pairing}

In this section we recall Casselman's canonical pairing, which plays a key role in second adjointness.

\subsection{Definition of the Casselman-Jacquet functor}

Recall the Casselman-Jacquet functor $$ \calJ_I : \calM^a (\frakg , K) \to \calM^a (\frakg , K_I N_{(I)})$$ given by\footnote{By $\lim_{k \in \bbZ_{\ge 1}}$ we understand the inverse limit, where the transition maps $V / (\frakn_{(I)}^-)^{k+1} V \to V / (\frakn_{(I)}^-)^k V$ are the standard projections.} $$ \calJ_I (V) := \left( \lim_{k \in \bbZ_{\ge 1}} (V / (\frakn_{(I)}^-)^k V) \right)^{K_I\textnormal{-finite} , \ \frakn_{(I)}\textnormal{-torsion}} .$$

\begin{lemma}\label{lem J is OK}
	The formula given for $\calJ_I$ indeed defines a functor as stated. One also has the description: $$ \left( \lim_{k \in \bbZ_{\ge 1}} (V / (\frakn_{(I)}^-)^k V) \right)^{K_I\textnormal{-finite} , \ \frakn_{(I)}\textnormal{-torsion}} = \left( \lim_{k \in \bbZ_{\ge 1}} (V / (\frakn_{(I)}^-)^k V) \right)^{\fraka_{cent,I}\textnormal{-finite}}.$$
\end{lemma}

\begin{proof}
	This is well-known, and not hard to establish based on all the admissibility Lemmas of this paper, so left as an exercise.
\end{proof}

The following is a basic fact proved by Casselman:

\begin{proposition}[Casselman]
	The functor $$ \calJ_I : \calM^a (\frakg , K) \to \calM^a (\frakg , K_I N_{(I)})$$ is exact.
\end{proposition}

\begin{proof}
	See, for example, \cite[\S 4.1.5]{Wa2} (the case $I = \emptyset$ is considered there, but the general case is completely analogous).
\end{proof}

\medskip

Analogously one has the functor $$ \calJ_I^- : \calM^a (\frakg , K) \to \calM^a (\frakg , K_I N_{(I)}^-) $$ (where one swaps the opposite parabolics).

\subsection{The canonical pairing}

\begin{theorem}[Casselman's canonical pairing]\label{thm Cass pair}
	Let $V \in \calM^a (\frakg , K)$. There exists a canonical $(\frakg , K_I)$-invariant pairing \begin{equation}\label{eq pairing} \calJ_I (V) \otimes \calJ_I^- (V^{\vee}) \to \bbC. \end{equation} Moreover, this pairing is non-degenerate; This means that the induced map $$ \calJ_I^- (V^{\vee}) \to \calJ_I (V)^{\vee}$$ is an isomorphism.
\end{theorem}

\begin{remark}
	The construction of the pairing of Theorem \ref{thm Cass pair} is analytical. We hope to have an algebraic treatment in the future. In the works \cite{ChGaYo}, \cite{GaYo} some conjectural algebraic (or algebro-geometric) reformulations are given.
\end{remark}

\begin{proof}[Proof (of Theorem \ref{thm Cass pair}).]
	Let us recall the construction of the pairing, due to Casselman, and provide a reference for the proof of non-degeneracy.
	
	\medskip
	
	For $v \in V$ and $\alpha \in V^{\vee}$ one has a corresponding matrix coefficient $$ m_{v , \alpha} \in C^{\infty} (G(\bbR)).$$ It has a convergent expansion \begin{equation}\label{eq expansion} m_{v , \alpha} (e^H) = \sum_{\lambda}  e^{\lambda (H)} \cdot p_{\lambda} (H) \quad  \quad (H \in \fraka_{cent,I})\end{equation} where $\lambda$ runs over a subset of $\fraka^*_{\bbC}$ of the form $$ \textnormal{finite subset} - \sum_{\alpha \in \Sigma} \bbZ_{\ge 0} \cdot \alpha,$$ and $p_{\lambda} (H)$ are polynomials. Let us denote by $Exp (\fraka_{cent,I})$ the space of formal expressions as the sum in (\ref{eq expansion}). We have the subspace $Exp^{fin} (\fraka_{cent,I}) \subset Exp(\fraka_{cent,I})$ consisting of finite sums. Then the asymptotic expansion of matrix coefficients is a map $$ V \otimes V^{\vee} \to Exp(\fraka_{cent,I}).$$ Completely formally (by ``continuity") this extends to a map $$ \lim_{k \in \bbZ_{\ge 1}} (V / (\frakn_{(I)}^-)^k V) \otimes \lim_{k \in \bbZ_{\ge 1}} (V^{\vee} / (\frakn_{(I)})^k V^{\vee}) \to Exp(\fraka_{cent,I})$$ and then restricts to a map $$ \calJ_I (V) \otimes \calJ_I^- (V^{\vee}) \to Exp^{fin} (\fraka_{cent,I}).$$ Composing with the map $$ Exp^{fin} (\fraka_{cent,I}) \to \bbC$$ given by evaluation at $0 \in H$, one obtains the desired pairing (\ref{eq pairing}).
	
	\medskip
	
	That this pairing is non-degenerate is non-trivial, first proven by Milicic (\cite{Mi}) for $I = \emptyset$, and then by Hecht and Schmid (\cite{HeSc}) in general.
	
\end{proof}

\section{Tempered admissible modules}\label{sec second adj}

In this section we describe the functor $temppres_I$ of tempered parabolic restriction, Casselman's canonical pairing for tempered admissible modules, and second adjointness for tempered admissible modules.

\subsection{The functors $pres_I$ and $pind_I$}\label{sec functors pind pres}

\begin{definition}\
	\begin{enumerate}
		\item We define the \emph{parabolic restriction} functor $$ pres_I : \calM (\frakg , K) \to \calM (\frakg_I , K_I) $$ by $$ pres_I (V) := \bbC_{- \rho_{(I)}} \otimes V / \frakn_{(I)} V.$$
		\item We define the \emph{parabolic induction} functor $$ \calM (\frakg , K) \leftarrow \calM (\frakg_I , K_I) : pind_I$$ as the right adjoint of the functor $pres_I$.
	\end{enumerate}
\end{definition}

\begin{remark}
	It is easy to see that the right adjoint $pind_I$ exists abstractly, by an adjoint functor theorem. Alternatively, one might interpret the relation $pind_I \cong \calB_I \circ \Delta_I$ which we will prove later (Proposition \ref{prop rel B Delta pind}) as a concrete description of the functor $pind_I$, which in particular shows its existence.
\end{remark}

\begin{remark}
	We similarly denote by $$ pres_I^- : \calM (\frakg , K) \to \calM (\frakg_I , K_I)$$ the functor analogous to $pres_I$, where we use $\frakn_{(I)}^-$ instead of $\frakn_{(I)}$; Thus $$ pres_I^- (V) := \bbC_{\rho_{(I)}} \otimes V / \frakn_{(I)}^- V.$$
\end{remark}

\begin{lemma}\label{lem pind pres adm}
	The functors $pind_I , pres_I$ preserve the subcategories of admissible modules.
\end{lemma}

\begin{proof}
	Given in Appendix \ref{sec proofs}.
\end{proof}

\subsection{Definition of tempered admissible modules}

\begin{definition}
	A module $V \in \calM^a (\frakg , K)$ is called \emph{tempered}, if all $\lambda \in wt_{\fraka} (pres_{\emptyset} (V))$ satisfy\footnote{By $\Re (\cdot)$ we denote the real part of a complex-valued functional on a real vector space.} $ \Re (\lambda) \ge_{\Sigma} 0$. We denote by $$ \calM^a (\frakg , K)_{\temp} \subset \calM^a (\frakg , K)$$ the full subcategory consisting of tempered modules.
\end{definition}

\begin{remark}
	By considering the symmetry given by $\dot{w}_0 \in K$, one can reformulate the above definition as: $V$ is tempered if all $\lambda \in wt_{\fraka} (pres_{\emptyset}^- (V))$ satisfy $ \Re (\lambda) \leq_{\Sigma} 0$.
\end{remark}

\begin{remark}
	It is known that under the Casselman-Wallach equivalence (Theorem \ref{Casselman-Wallach}), tempered modules as defined above match with tempered representations in the usual sense (in particular, as used in \cite{CrHi}). See, for example, \cite{Ba} and in particular \cite[Lemma 4.4]{Ba}, or \cite{Yo2} and in particular \cite[Lemma 5.4.2]{Yo2}, where the relation of the above definition of temperedness with decay of matrix coefficients is discussed.
\end{remark}

\begin{remark}\label{rem tempered then central}
	Let $W \in \calM^a (\frakg_I , K_I)$ be tempered. Then, in particular, all $\omega \in wt_{\fraka_{cent,I}} (W)$ satisfy: $$ \Re (\omega) = 0.$$
\end{remark}

\subsection{Parabolic induction and restriction in the tempered case}

\begin{remark}
	The parabolic induction of a tempered module is tempered (see, for example, \cite[Corollary 5.5.2]{Yo2}), and the contragradient of a tempered module is tempered (see, for example, \cite[Corollary 6.1.6]{Yo2}). However, the parabolic restriction of a tempered module is not necessarily tempered.
\end{remark}

In view of the last Remark, let us define:

\begin{definition}
	We define $$ temppres_I : \calM^a (\frakg , K)_{\temp} \to \calM^a (\frakg_I , K_I)_{\temp}$$ to be the left adjoint of $$ \calM^a (\frakg , K)_{\temp} \xleftarrow{} \calM^a (\frakg_I , K_I)_{\temp} : pind_I. $$
\end{definition}

\begin{remark}
	Of course, we similarly define $$ temppres_I^- : \calM^a (\frakg , K)_{\temp} \to \calM^a (\frakg_I , K_I)_{\temp},$$ using the opposite parabolic.
\end{remark}

Let us now describe $temppres_I$ more concretely (and thus, in particular, deduce its existence). We denote by $$ \calM^a (\frakg_I , K_I)_{\Sigma\textnormal{-}\temp} \subset \calM^a (\frakg_I , K_I)$$ the full subcategory consisting of modules $W$ for which one has $\Re (\lambda) \leq_{\Sigma} 0$ for all $\lambda \in wt_{\fraka} (pres_{\emptyset}^- (W))$ (notice, in contrast, that the condition for the $(\frakg_I , K_I)$-module $W$ to be tempered is $\Re (\lambda) \leq_{I} 0$ for all $\lambda \in wt_{\fraka} (pres_{\emptyset}^- (W))$). We then have $$ \calM^a (\frakg_I , K_I)_{\temp} \subset \calM^a (\frakg_I , K_I)_{\Sigma\textnormal{-}\temp}$$ (because $\leq_{I}$ is a finer partial order than $\leq_{\Sigma}$) and also $$ pres_I (\calM^a (\frakg , K)_{\temp}) \subset \calM^a (\frakg_I , K_I)_{\Sigma\textnormal{-}\temp}$$ (by the transitivity of parabolic restriction).

\begin{notation}\label{notation}
	In what follows it will be convenient, given $W$ which lies in $\calM^a (\frakg_I , K_I)$ or in $\calM^a (\frakg , K_I N_{(I)})$ and given $\lambda \in \fraka^*$, to denote $$ W_{\langle \lambda \rangle} := \bigoplus_{\omega \in (\fraka_{cent,I})^*_{\bbC} \textnormal{ s.t. } \Re (\omega) = \lambda|_{\fraka_{cent,I}}} W_{\fraka_{cent,I} , \omega}$$ (this is a $(\frakg_I , K_I)$-module).
\end{notation}

\begin{lemma}
	A module $W \in \calM^a (\frakg_I , K_I)_{\Sigma\textnormal{-}\temp}$ lies in $\calM^a (\frakg_I , K_I)_{\temp}$ if and only if $W = W_{\langle 0 \rangle}$.
\end{lemma}

\begin{proof}
	Notice that, in view of Casselman's submodule theorem, one has $$ wt_{\fraka_{cent,I}} (W) = wt_{\fraka} (pres^-_{\emptyset} (W))|_{\fraka_{cent,I}}.$$ Therefore, we need to check that for $\lambda \in wt_{\fraka} (pres^-_{\emptyset} (W))$, one has $\Re (\lambda|_{\fraka_{cent,I}}) = 0$ if and only if $\lambda \leq_{I} 0$. So, we are reduced to chekcing that for $\lambda \in \fraka^*$ satisfying $\lambda \leq_{\Sigma} 0$, one has $\lambda \leq_I 0$ if and only if $\lambda|_{\fraka_{cent,I}} = 0$. This is clear in view of the equality $\fraka^{+,I} = \fraka^{+,\Sigma} + \fraka_{cent,I}$.
\end{proof}

From the last Lemma we see that the functor $$(\cdot)_{\langle 0 \rangle} : \calM^a (\frakg_I , K_I)_{\Sigma\textnormal{-}\temp} \to \calM^a (\frakg_I , K_I)_{\temp}$$ is both the right and the left adjoint of the inclusion $$ \calM^a (\frakg_I , K_I)_{\temp} \subset \calM^a (\frakg_I , K_I)_{\Sigma\textnormal{-}\temp}.$$ We thus conclude:

\begin{claim}
	One has  $$ temppres_I = (\cdot)_{\langle 0 \rangle}\circ pres_I : \calM^a (\frakg , K)_{\temp} \to \calM^a (\frakg_I , K_I)_{\temp}.$$
\end{claim}

\subsection{Exactness}

The following lemma has a simple proof, but it is key.

\begin{lemma}\label{lem tempered pres J same}
	Let $V \in \calM^a (\frakg , K)_{\temp}$. The projection map $$ \calJ_I^- (V) \to \bbC_{\rho_{(I)}} \otimes pres_I (V)$$ induces an isomorphism $$ \calJ_I^- (V)_{\langle \rho_{(I)} \rangle} \to pres_I (V)_{\langle 0 \rangle} = temppres_I (V).$$
\end{lemma}

\begin{proof}
	One needs to see that $$ \calJ_I^- (V)_{\langle \rho_{(I)} \rangle} \to pres_I (V)_{\langle 0 \rangle} $$ is injective. This will follow if we see that $$ \left( \frakn_{(I)} \calJ_I^- (V) \right)_{\langle \rho_{(I)} \rangle} = 0.$$ To that end, notice that all $\omega \in wt_{\fraka_{cent,I}} (\frakn_{(I)} \calJ_I^- (V))$ are contained in $$ wt_{\fraka_{cent,I}}(V / \frakn_{(I)} V) + \left. \left( \left( \sum_{\alpha \in R^+_{(I)}} \bbZ_{\ge 0} \cdot \alpha \right)  \setminus \{ 0 \} \right)\right\rvert_{\fraka_{cent,I}}.$$ Thus, since $V$ is tempered, the real part of every $\omega \in wt_{\fraka_{cent,I}} (\frakn_{(I)} \calJ_I^- (V))$ is the restriction to $\fraka_{cent,I}$ of some weight of the form $$ \rho_{(I)} + \lambda + \left( \left( \sum_{\alpha \in R^+_{(I)}} \bbZ_{\ge 0} \cdot \alpha \right)  \setminus \{ 0 \} \right)$$ where $\lambda \in \fraka^*$ satisfies $\lambda \ge_{\Sigma} 0$. In particular, this real part clearly can not be $\rho_{(I)}|_{\fraka_{cent,I}}$.
\end{proof}

The following can be thought of as the main difference between the tempered and non-tempered cases, explaining why the archimedean tempered case re-gains similarity to the non-archimedean case.

\begin{proposition}\label{temppres exact}
	The functor $$ temppres_I : \calM^a (\frakg , K)_{\temp} \to \calM^a (\frakg_I , K_I)_{\temp}$$ is exact.
\end{proposition}

\begin{proof}
	This follows immediately from Lemma \ref{lem tempered pres J same}, as $\calJ_I^-$ is exact.
\end{proof}

\subsection{Casselman's canonical pairing for tempered admissible modules}

Let $V \in \calM^a (\frakg , K)$. Recall (Theorem \ref{thm Cass pair}) Casselman's canonical pairing $$ \calJ_I (V) \otimes \calJ_I^- (V) \to \bbC.$$ It induces a pairing $$ \calJ_I (V)_{\langle \rho_{(I)} \rangle} \otimes \calJ_I^- (V)_{\langle -\rho_{(I)} \rangle} \to \bbC.$$ Since the former pairing is non-degenerate, so is the latter. Now, assume that $V$ is tempered. By Lemma \ref{lem tempered pres J same} the latter pairing can be rewritten as $$ temppres_I^- (V) \otimes temppres_I (V^{\vee}) \to \bbC.$$

\medskip

Let us thus summarize:

\begin{theorem}\label{thm Cass pair tempered}
	Let $V \in \calM^a (\frakg , K)_{\temp}$. There exists a canonical non-degenerate pairing $$ temppres_I^- (V) \otimes temppres_I (V^{\vee}) \to \bbC.$$ In other words, one has a canonical isomorphism $$ temppres_I^- (V)^{\vee} \cong temppres_I (V^{\vee}).$$
\end{theorem}

\subsection{Second adjointness for tempered admissible modules}

One can quite formally rewrite Theorem \ref{thm Cass pair tempered} as follows:

\begin{theorem}\label{thm second adj temp}
	There is a natural adjunction $$ pind_I : \calM^a (\frakg_I , K_I)_{\temp} \rightleftarrows \calM^a (\frakg , K)_{\temp} : temppres_I^-. $$
\end{theorem}

\begin{proof}
	Let $V \in \calM^a (\frakg , K)_{\temp}$ and $W \in \calM^a (\frakg_I , K_I)_{\temp}$. One has: $$ \Hom (pind_I (W) , V) \cong \Hom (V^{\vee} , pind_I (W)^{\vee}) \cong \Hom (V^{\vee} , pind_I (W^{\vee})) \cong $$ $$ \Hom (temppres_I (V^{\vee}) , W^{\vee}) \cong \Hom (temppres_I^- (V)^{\vee} , W^{\vee}) \cong \Hom (W , temppres_I^- (V)).$$
	
	Here, we used the well-known isomorphism $pind_I (W)^{\vee} \cong pind_I (W^{\vee})$.
\end{proof}

\section{Relation to Bernstein morphisms}\label{sec Bernstein morphisms}

In this section we briefly record how the canonical pairing for tempered admissible modules should be related to the construction of Bernstein morphisms.

\subsection{Boundary degenerations and Bernstein morphisms}

Let us denote $$ Y_I := (G(\bbR) \times G(\bbR)) / (\Delta G_I (\bbR) \cdot (N_{(I)}^- (\bbR) \times N_{(I)} (\bbR))).$$ One has $Y_{\Sigma} \cong G(\bbR)$, and the $Y_I$'s are ``boundary degenerations" of $Y_{\Sigma}$. According to ideas of J. Bernstein, one should have ($G(\bbR)\times G(\bbR)$-equivariant) \emph{Bernstein morphisms} $$ Ber_I : L^2 (Y_I) \to L^2 (Y_{\Sigma}),$$ which are (not necessarily surjective) isometries, and which should provide a conceptual derivation of the Plancherel formula for $L^2(Y_{\Sigma})$ (modulo knowledge of twisted discrete spectrum).

\medskip

And indeed, such Bernstein morphisms (in a much greater generality, of some spherical varieties) were constructed in \cite{DeKnKrSc} (see also \cite{SaVe} for the non-archimedean case).

\subsection{Relation of the canonical pairing to boundary degenerations}\label{sec Bernstein morphisms description of prime}

The following is a (presumably) well-known ``automatic continuity" result:

\begin{lemma}\label{lem auto cont}
	Let $U \in \calM^a (\frakg \oplus \frakg , K \times K)$. Then the map\footnote{By $W^{*,\frakh}$ we denote the space of functionals on $W$ which are invariant under $\frakh$, i.e. annihilating $\frakh W$.} $$ \Hom_{\frakg \oplus \frakg , K \times K} (U , C^{\infty}(Y_I)) \to U^{* , \Delta \frakg_I + (\frakn_{(I)}^- \oplus \frakn_{(I)})}$$ given by evaluation at $1$ is a bijection.
\end{lemma}

\begin{proof}
	Given in Appendix \ref{sec proofs}.
\end{proof}

Let $V \in \calM^a (\frakg , K)_{\temp}$. Corresponding to the tautological pairing $V \otimes V^{\vee} \to \bbC$, under the identification of Lemma \ref{lem auto cont}, is the matrix coefficients map $$ \alpha : V \otimes V^{\vee} \to C^{\infty} (Y_{\Sigma}).$$ Additionally, corresponding to the pairing of Theorem \ref{thm Cass pair tempered}, again using Lemma \ref{lem auto cont}, one obtains a map $$\alpha^{\prime} : V \otimes V^{\vee} \to C^{\infty} (Y_I).$$

\subsection{Bernstein morphism via canonical pairing}

Let us fix a Plancherel decomposition for $L^2 (Y_{\Sigma})$ (see \cite{Be} for more details): A measure space $(\Omega , \mu)$, and for each $\omega \in \Omega$ a tempered irreducible module $V_{\omega} \in \calM^a (\frakg , K)_{\temp}$. The matrix coefficient map $$ \alpha_{\omega} : V_{\omega} \otimes V_{\omega}^{\vee} \to C^{\infty} (Y_{\Sigma})$$ gives rise to the ``adjoint" map $$ \beta_{\omega} : C^{\infty}_c (Y_{\Sigma}) \to (V_{\omega} \otimes V_{\omega}^{\vee})^{(2)}$$ (here $(\cdot)^{(2)}$ denotes the completion w.r.t. the inner product - $V_{\omega} \otimes V_{\omega}^{\vee}$ has a canonical one) - again, see \cite{Be} for details. The data is required to give rise to an isomorphism of Hilbert spaces $$ pl: L^2 (Y_{\Sigma}) \xrightarrow{\cong} \int_{\omega \in \Omega}^{\oplus} (V_{\omega} \otimes V_{\omega}^{\vee})^{(2)} d\mu : \quad \phi \mapsto \left[ \beta_{\omega} (\phi) \right]_{\omega \in \Omega}.$$

\medskip

Now, by \S\ref{sec Bernstein morphisms description of prime} we also have maps $$ \alpha_{\omega}^{\prime} : V_{\omega} \otimes V_{\omega}^{\vee} \to C^{\infty} (Y_I),$$ and to them correspond the ``adjoint" maps $$\beta_{\omega}^{\prime} : C^{\infty}_c (Y_I) \to (V_{\omega} \otimes V_{\omega}^{\vee})^{(2)}.$$

\begin{expectation}
	The Bernstein morphism $$ Ber_I : L^2 (Y_I) \to L^2 (Y_{\Sigma}) $$ is given by $$ \phi \mapsto pl^{-1} (\left[ \beta_{\omega}^{\prime} (\phi) \right]_{\omega \in \Omega})$$ for $\phi \in C^{\infty}_c (Y_I)$.
\end{expectation}

\begin{remark}
	As far as we understand, establishing the above expectation should be simply a matter of comparing the languages of \cite{DeKnKrSc} and the current paper.
\end{remark}

\section{Functors}\label{sec functors}

In this section, we describe the functors $\calB_I$ and $\calC_I$ which we studied in \cite{Yo1}, and their relation with $pind_I$ and $pres_I$.

\medskip

One can summarize the functors in the following diagram: $$ 
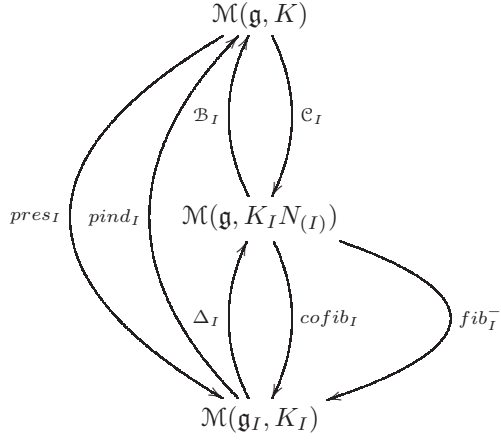
\xymatrix@R=60pt{\calM (\frakg , K) \ar@/^1pc/[d]^{\calC_I} \ar@/_6pc/[dd]_{pres_I} \\ \calM (\frakg , K_I N_{(I)}) \ar@/^1pc/[u]^{\calB_I} \ar@/^1pc/[d]^{cofib_I} \ar@/^6pc/[d]^{fib^-_I} \\ \calM (\frakg_I , K_I) \ar@/^1pc/[u]^{\Delta_I} \ar@/^3.5pc/[uu]^{pind_I}}.$$ Here, all functors preserve the admissible subcategories. We have three adjunctions $$ (\calB_I , \calC_I); \quad (\Delta_I , cofib_I); \quad (pres_I , pind_I),$$ the relation $$ pind_I \cong \calB_I \circ \Delta_I,$$ a morphism $$ cofib_I \to fib_I^-, $$ and on the admissible subcategories an isomorphism $$ fib^-_I \circ \calC_I \cong pres^-_I$$ (where $pres_I^-$ is analogous to $pres_I$, but using the opposite parabolic).

\subsection{The functors $\calB_I$ and $\calC_I$}

\begin{definition}\
	\begin{enumerate}
		\item We define the functor\footnote{This can be called \emph{Bernstein's functor}, as it is similar to a functor Bernstein has studied, which in turn is a version of \emph{Zuckerman's functor}.} $$\calB_I : \calM (\frakg , K_I N_{(I)}) \to \calM (\frakg , K)$$ by $$ \calB_I (V) := \left( \calO (K) \otimes V \right)^{K_I}_{\frakk}.$$ Here the notations are as follows. The $K$-action on $\calO (K) \otimes V$ is the left regular one on $\calO (K)$. The $\frakg$-action on $\calO (K) \otimes V$ is $\xi (f) (k) = {}^{k^{-1}} \xi \cdot f(k)$, where we think about $f \in Fun (K,V) \cong \calO (K) \otimes V$ . The $K_I$-action w.r.t. which we take invariants is $ m(f \otimes v) = R_m f \otimes mv$ (here $R_m$ denotes the right regular action of $m$). The $\frakk$-action w.r.t. which we take coinvariants is the difference between the $\frakk$-action gotten by differentiating the $K$-action, and the $\frakk$-action gotten by restricting the $\frakg$-action. The actions of $\frakg$ and $K$ are well-defined after passing to the invariants and coinvariants, and we obtain a $(\frakg , K)$-module in this way.
		\item We define the functor\footnote{This can be called the \emph{Casselman-Jacquet functor}, in veiw of Theorem \ref{thm Cass pair reformulation}.} $$ \calM (\frakg , K_I N_{(I)}) \leftarrow \calM (\frakg , K): \calC_I$$ as the right adjoint of $\calB_I$.
	\end{enumerate}
\end{definition}

\begin{remark}
	In more geometric terms, say using $D$-algebras, the functor $\calB_I$ is given by forgetting the $N_{(I)}$-equivariancy, followed by performing $*$-averaging from $K_I$-equivariancy to $K$-equivariancy. See \cite{Yo1} for this as well as a more detailed (although, at some points, yet premature) discussion of the functors $\calB_I$ and $\calC_I$.
\end{remark}

\begin{remark}\label{rem concrete C}
	Let us describe the functor $\calC_I$ more concretely (again, see \cite{Yo1} for details). It is given by $$ \calC_I (V) = \left( \prod_{\alpha} V^{[\alpha]} \right)^{K_I\textnormal{-finite} , \  \frakn_{(I)}\textnormal{-torsion}}.$$
\end{remark}

\begin{lemma}\label{lem B C adm}
	The functors $\calB_I , \calC_I$ preserve the subcategories of admissible modules.
\end{lemma}

\begin{proof}
	Given in Appendix \ref{sec proofs}.
\end{proof}

\subsection{The functors $\Delta_I$, $cofib_I$ and $fib^-_I$}

\begin{definition}\
	\begin{enumerate}
		\item We define the functor $$ \calM (\frakg_I , K_I) \leftarrow \calM (\frakg , K_I N_{(I)}) : cofib_I$$ by $$ cofib_I (V) := \bbC_{\rho_{(I)}} \otimes V^{\frakn_{(I)}}.$$
		\item We define the functor $$ \Delta_I : \calM (\frakg_I , K_I) \to \calM (\frakg , K_I N_{(I)})$$ as the left adjoint of $cofib_I$.
		\item We define the functor $$ fib^-_I: \calM (\frakg , K_I N_{(I)}) \to \calM (\frakg_I , K_I)$$ by $$fib^-_I (V) := \bbC_{\rho_{(I)}} \otimes V / \frakn_{(I)}^- V.$$
	\end{enumerate}
\end{definition}

\begin{remark}
	Let us describe the functor $\Delta_I$ more concretely. It is given by $$ \Delta_I (V) := U(\frakg) \underset{U(\frakg_I + \frakn_{(I)})}{\otimes} (\bbC_{-\rho_{(I)}} \otimes V),$$ where $\bbC_{-\rho_{(I)}} \otimes V$ is considered as a $U(\frakg_I + \frakn_{(I)})$-module by making $\frakn_{(I)}$ act by zero.
\end{remark}

\begin{remark}
	Notice that we have a morphism $$ cofib_I \to fib_I^-,$$ given by $V^{\frakn_{(I)}} \hookrightarrow V \twoheadrightarrow V / \frakn_{(I)}^- V$.
\end{remark}

\begin{lemma}\label{lem cofib delta fib adm}
	The functors $cofib_I , \Delta_I , fib^-_I$ preserve the subcategories of admissible modules.
\end{lemma}

\begin{proof}
	Given in Appendix \ref{sec proofs}.
\end{proof}

\begin{proposition}\label{prop rel B Delta pind}
	One has $$ \calB_I  \circ \Delta_I \cong pind_I.$$
\end{proposition}

\begin{proof}
	One first checks that the map $$ \left( \calO (K) \otimes V \right)^{K_I} \to \left( \calO (K) \otimes (U(\frakg) \underset{U(\frakg_I + \frakn_{(I)})}{\otimes} V) \right)^{K_I}_{\frakk},$$ given by inserting $1$ at the $U(\frakg)$-component, is an isomorphism of $K$-representations (this is the analog of the ``compact picture" for parabolic induction).
	
	\medskip
	
	Composing the inverse of this isomorphism with the evaluation at $1 \in K$, we obtain a map $$ \left( \calO (K) \otimes (U(\frakg) \underset{U(\frakg_I + \frakn_{(I)})}{\otimes} V) \right)^{K_I}_{\frakk} \to V.$$ One now routinely checks that for a $(\frakg , K)$-module $W$, by composing with this map one obtains a bijection $$ \textnormal{Hom}_{\frakg , K} (W , \left( \calO (K) \otimes (U(\frakg) \underset{U(\frakg_I + \frakn_{(I)})}{\otimes} V) \right)^{K_I}_{\frakk}) \cong \textnormal{Hom}_{\frakg_I , K_I} (W / \frakn_{(I)} W , \bbC_{2 \rho_{(I)}} \otimes V).$$
\end{proof}

\subsection{Casselman's canonical pairing in terms of the functors}

Casselman's canonical pairing (Theorem \ref{thm Cass pair}) has the following reformulation:

\begin{theorem}\label{thm Cass pair reformulation}
	There exists a canonical isomorphism of functors $$ \calC_I \cong \calJ_I : \calM^a (\frakg , K) \to \calM^a (\frakg , K_I N_{(I)}).$$
\end{theorem}

\begin{proof}
	This will be clearly a reformulation of Theorem \ref{thm Cass pair}, once we establish an isomorphism $$ \calC_I \cong (\cdot)^{\vee} \circ \calJ^-_I \circ (\cdot)^{\vee} : \calM^a (\frakg , K) \to \calM^a (\frakg , K_I N_{(I)}). $$ This is estbalished using the concrete description of $\calC_I$ in Remark \ref{rem concrete C}. Indeed, clearly $$ \prod_{\alpha} V^{[\alpha]} \cong (V^{\vee})^*,$$ and then $$ \left( \prod_{\alpha} V^{[\alpha]} \right)^{\frakn_{(I)}\textnormal{-torsion}} \cong (V^{\vee})^{* , \frakn_{(I)}\textnormal{-torsion}} \cong \left( \lim_{k \in \bbZ_{\ge 1}} (V^{\vee} / \frakn_{(I)}^k V^{\vee}) \right)^{\star} \cong \ldots $$ where $(\lim V^{\vee} / \frakn_{(I)}^k V^{\vee} )^{\star}$ denotes the subspace of the space of functionals, consisting of those which factor through the projection onto one of the $V^{\vee} / \frakn_{(I)}^k V^{\vee}$'s. By $\fraka_{cent,I}$-weight consideration, we can continue: $$ \cong \left( \left( \lim_{k \in \bbZ_{\ge 1}} (V^{\vee} / \frakn_{(I)}^k V^{\vee}) \right)^{\fraka_{cent,I}\textnormal{-finite}} \right)^{* , \ \fraka_{cent,I}\textnormal{-finite}}.$$ Therefore, we obtain: $$ \calC_I (V) = \left( \prod_{\alpha} V^{[\alpha]} \right)^{K_I\textnormal{-finite}, \ \frakn_{(I)}\textnormal{-torsion}} \cong $$ $$ \cong \left( \left( \lim_{k \in \bbZ_{\ge 1}} (V^{\vee} / \frakn_{(I)}^k V^{\vee}) \right)^{\fraka_{cent,I}\textnormal{-finite}} \right)^{* , \ K_I\textnormal{-finite} \ \fraka_{cent,I}\textnormal{-finite}} \cong \calJ_I^- (V^{\vee})^{\vee}.$$ 
\end{proof}

For our current purposes, only the following corollary will be needed:

\begin{corollary}\label{cor of C_I=J_I}
	One has an isomorphism of functors $$ fib^-_I \circ \calC_I \cong pres^-_I : \calM^a (\frakg , K) \to \calM^a (\frakg_I , K_I).$$
\end{corollary}

\begin{proof}
	In view of Theorem \ref{thm Cass pair reformulation}, this follows from the easy relation $$ fib^-_I \circ \calJ_I \cong pres^-_I.$$
\end{proof}

\section{Second adjointness - second take}\label{sec second adj second take}

In this section we describe again second adjointness for tempered admissible modules, but with an emphasis on trying to work with all admissible modules (rather than just the tempered ones).

\subsection{Second ``preadjointness" for admissible modules}

From \S\ref{sec functors} wee see that we have an adjunction $$ pind_I : \calM^a (\frakg_I , K_I) \rightleftarrows \calM^a (\frakg , K) : cofib_I \circ \calC_I,$$ and a morphism $$ cofib_I \to fib_I^-.$$ Thus, we obtain a morphism of functors $$ cofib_I \circ \calC_I \to fib^-_I \circ \calC_I \cong pres^-_I,$$ where the latter isomorphism is Corollary \ref{cor of C_I=J_I} which, let us remind, uses the non-trivial Casselman's canonical pairing (Theorem \ref{thm Cass pair reformulation}). We see that the failure of the naive second adjointness, that is, of $(pind_I , pres_I^-)$ being an adjoint pair, is encoded by the non-isomorphicity of $cofib_I \to fib^-_I$. Nevertheless, we have a ``candidate for a unit" for an adjunction between $pind_I$ and $pres^-_I$, namely the composition $$ \textnormal{Id} \to (cofib_I \circ \calC_I) \circ pind_I \to pres^-_I \circ pind_I.$$ In other words, we have maps \begin{equation}\label{eq preadjunction} \textnormal{Hom} (pind_I (W) , V) \to \textnormal{Hom} (W , pres_I^- (V)) \end{equation} functorial in $W \in \calM^a (\frakg_I , K_I)$ and $V \in \calM^a (\frakg , K)$. One might call this the second ``preadjointness".

\subsection{Second adjointness for tempered admissible modules}

\begin{claim}\label{lem preadj isom}
	Let $V \in \calM^a (\frakg , K)_{\temp}$ and let $W \in \calM^a (\frakg_I , K_I)$ be such that\footnote{Here recall Notation \ref{notation}.} $W = W_{\langle 0 \rangle}$. Then the morphism (\ref{eq preadjunction}) is an isomorphism.
\end{claim}

\begin{proof}
	It is enough to show that the map $$  \calC_I (V)^{\frakn_{(I)}} \to \calC_I (V) / \frakn_{(I)}^- \calC_I (V)$$ induces an isomorphism $$ (\calC_I (V)^{\frakn_{(I)}})_{\langle -\rho_{(I)} \rangle} \to (\calC_I (V) / \frakn_{(I)}^- \calC_I (V))_{\langle -\rho_{(I)}\rangle}$$ (recall Notation \ref{notation}). In fact, decomposing this map as $$ \calC_I (V)^{\frakn_{(I)}} \hookrightarrow \calC_I (V) \twoheadrightarrow \calC_I (V) / \frakn_{(I)}^- \calC_I (V),$$ we will see that these two maps separately become an isomorphism after applying $(\cdot)_{\langle -\rho_{(I)} \rangle}$.
	
	\medskip
	
	Let us argue by contradiction, assuming that one of these two isomorphisms fails. Then it is easy to see that there exists $\omega \in wt_{\fraka_{cent,I}} (\calC_I (V) / \frakn_{(I)}^- \calC_I (V)) $ such that $$\Re (\omega) \in \left( -\rho_{(I)} + \sum_{\alpha \in R^+_{(I)}} \bbZ_{\ge 0} \cdot \alpha \right) \setminus \{ -\rho_{(I)} \}$$ (here in the right hand side we understand restrictions to $\fraka_{cent,I}$). Then, by Casselman's submodule theorem, there will exist $\lambda \in wt_{\fraka} (\calC_I (V) / \frakn_{(\emptyset)}^- \calC_I (V))$ such that $\lambda|_{\fraka_{cent,I}} = \omega$. In other words, there will exist $\lambda^{\prime} \in wt_{\fraka} (pres^-_{(\emptyset)} (V))$ such that $$\Re (\lambda^{\prime})|_{\fraka_{cent,I}} \in \left( \sum_{\alpha \in R^+_{(I)}} \bbZ_{\ge 0} \cdot \alpha \right) \setminus \{ 0 \};$$ Here we used $$pres_{\emptyset}^- (V) \cong pres_{\emptyset}^- (pres_I^- V) \overset{\textnormal{cor. } \ref{cor of C_I=J_I}}{\cong} pres_{\emptyset}^- (fib_I^- (\calC_I (V))) \cong  fib^-_{\emptyset} (\calC_I (V))$$ (where some of the functors where not formally defined with their current domain, but their meaning is completely clear). But clearly then $\Re (\lambda^{\prime}) \leq_{\Sigma} 0$ \underline{does not} hold, contradicting $V$ being tempered.
\end{proof}

\begin{corollary}\label{cor second adjointness tempered}
	The preadjointness morphism (\ref{eq preadjunction}) is an isomorphism when $V$ and $W$ are tempered.
\end{corollary}

\begin{proof}
	This follows from Claim \ref{lem preadj isom} because, in view of Remark \ref{rem tempered then central}, if $W$ is tempered then $W = W_{\langle 0\rangle}$.
\end{proof}

Notice, finally, that Corollary \ref{cor second adjointness tempered} gives one more proof of Theorem \ref{thm second adj temp}.

\appendix

\section{Proofs of Lemmas}\label{sec proofs}

\subsection{Proofs of the Lemmas characterizing admissibility}

\subsubsection{Proof of Lemma \ref{clm admissible 1}}\

	$(1) \implies (2):$ Since $V$ is finitely generated over $U(\frakg)$, a Theorem of Harish-Chandra (\cite[\S 3.4.1]{Wa2}) implies that each $V^{[\alpha]}$ is finitely generated over $Z(\frakg)$. Since $Z(\frakg)$ acts finitely on $V$, we deduce that each $V^{[\alpha]}$ is in fact finite-dimensional.

	$(2) \implies (3):$ Since $V$ is finitely generated over $U(\frakg)$, it is generated over $U(\frakg)$ by finitely many of the $V^{[\alpha]}$'s, so it is enough to show that $Z(\frakg)$ acts finitely on each $V^{]\alpha]}$. This, in turn, is clear since $Z(\frakg)$ preserves each $V^{[\alpha]}$ and each $V^{[\alpha]}$ is finite-dimensional by our assumption.
	
	$(3) \implies (4):$ This follows from the fact that there are, up to isomorphism, only finitely many irreducible $(\frakg , K)$-modules with a given infinitesimal character (for that fact, see \cite[\S 5.5.6]{Wa2}; Alternatively (and algebraically), it can be easily deduced from Beilinson-Bernstein localization theory). Indeed, that fact implies, since $V$ is $Z(\frakg)$-finite, that there exists a finite set $S$ of isomorphism classes of irreducible $(\frakg , K)$-modules such that the isomorphism class of every irreducible subquotient of $V$ lies in $S$. Then we can pick a finite set $T \subset \hat{K}$ such that for every irreducible $(\frakg , K)$-module $W$ of isomorphism class in $S$, one has $W^{[\alpha]} \neq 0$ for some $\alpha \in T$. Now, the functor from the category of $(\frakg , K)$-modules all of whose irreducible subquotients are of isomorphism class in $S$, to the category of vector spaces, given by $W \mapsto \oplus_{\alpha \in T} W^{[\alpha]}$, is exact, conservative (i.e. maps a non-zero object to a non-zero object), and the image of $V$ under it is of finite length (i.e. a finite-dimensional vector space). This implies that $V$ has finite length.

	$(4) \implies (1):$ One reduces immediately to the case when $V$ is irreducible. Then that $V$ is finitely generated over $U(\frakg)$ is clear. The center $Z(\frakg)$ acts finitely because it in fact acts by scalars, by Schur's Lemma (\cite[\S 0.5.2, \S 3.3.2]{Wa2}).

\subsubsection{Proof of Lemma \ref{clm admissible 2}}

Let us first assume that the $(\frakg , K_I N_{(I)})$-module $V$ is $Z(\frakg)$-finite and deduce some preliminary observations. One has the Harish-Chandra homomorphism $Z(\frakg) \to Z(\frakg_I)$, which is finite, and from its definition one sees that the action of $Z(\frakg)$ on $V^{\frakn_{(I)}}$ factors through this homomorphism. Therefore, we deduce that $V^{\frakn_{(I)}}$ is $Z(\frakg_I)$-finite, and hence $\fraka_{cent,I}$-finite. Considering, for $k \in \bbZ_{\ge 1}$, the exact sequence of $\frakg_I$-modules \begin{equation}\label{eq induction} 0 \to V^{ \frakn_{(I)}^k} \to V^{\frakn_{(I)}^{k+1}} \to \Hom_{\bbC} (\frakn_{(I)}^k , V^{\frakn_{(I)}} ) \end{equation} (where the last arrow is given by acting on $V^{\frakn_{(I)}^{k+1}}$ by $\frakn_{(I)}^k$), we by induction deduce that $V^{\frakn_{(I)}^k}$ are $\fraka_{cent,I}$-finite for all $k \in \bbZ_{\ge 1}$. In particular, $V$ is $\fraka_{cent,I}$-locally finite. Moreover, the above exact sequence shows that $$wt_{\fraka_{cent,I}} (V^{\frakn_{(I)}^{k+1}} / V^{\frakn_{(I)}^{k}}) \subset wt_{\fraka_{cent,I}} (V^{\frakn_{(I)}}) - k \cdot wt_{\fraka_{cent,I}} (\frakn_{(I)}).$$
	
	\medskip
	
	Now we will proceed with the steps.
	
	\medskip
	
	$(1) \implies (2):$ We remarked above that $V^{\frakn_{(I)}}$ is $Z(\frakg_I)$-finite. Since $V$ is finitely generated over $U(\frakg)$, there exists an $\fraka_{cent,I}$-stable finite-dimensional subspace $V_0 \subset V$ such that $V = U(\frakn_{(I)}^-) U(\frakg_I) V_0$. It is then clear by $\fraka_{cent,I}$-weight consideration that there exists $k \in \bbZ_{\ge 1}$ such that $V^{\frakn_{(I)}} \subset (\frakn_{(I)}^-)^k U(\frakg_I) V_0$. Therefore, as $(\frakn_{(I)}^-)^k U(\frakg_I) V_0$ is finitely generated over $U(\frakg_I)$, so is $V^{\frakn_{(I)}}$.
	
	$(2) \implies (3):$ One shows that $V^{\frakn_{(I)}^k}$ is admissible for any $k \in \bbZ_{\ge 1}$ by induction on $k$, using the exact sequence (\ref{eq induction}).

	$(3) \implies (4):$ Since $V = \cup_{k \in \bbZ_{\ge 1}} V^{\frakn_{(I)}^k}$, it is clear that $V$ is $\fraka_{cent,I}$-locally finite. It is enough now to show that for every $\lambda \in (\fraka_{cent,I})^*_{\bbC}$ there exists $k \in \bbZ_{\ge 1}$ such that $V_{\fraka_{cent,I},\lambda} \subset V^{\frakn_{(I)}^k}$. This is clear by $\fraka_{cent,I}$-weight consideration, from the last preliminary observation.

	$(4) \implies (5):$ Let $I \subset Z(\frakg)$ be an ideal of finite codimension that acts by zero on $V$. There exists, depending only on $I$, a finite set $S \subset (\fraka_{cent,I})^*_{\bbC}$ such that $wt_{\fraka_{cent,I}} (V^{\frakn_{(I)}}) \subset S$. Consider now the functor from the category of $(\frakg , K_I N_{(I)})$-modules on which $I$ acts by zero and which are $\fraka_{cent,I}$-locally finite, to the category of $(\frakg_I , K_I)$-modules, given by $W \mapsto \oplus_{\lambda \in S} W_{\fraka_{cent,I} , \lambda}$. This functor is exact, conservative, and the image of $V$ under it is of finite length. This implies that $V$ has finite length.
	
	$(5) \implies (1)$: One reduces immediately to the case when $V$ is irreducible. Then that $V$ is finitely generated over $U(\frakg)$ is clear. The center $Z(\frakg)$ acts finitely because it in fact acts by scalars, by Schur's Lemma  (\cite[\S 0.5.2]{Wa2}).

\subsubsection{Proof of Lemma \ref{clm admissible 3}}\

	$(1) \implies (2):$ This is clear, in view of the implication $(1) \implies (4)$ of Lemma \ref{clm admissible 2}, as well as the final preliminary observation in the proof of Lemma \ref{clm admissible 2}.

	$(2) \implies (1):$ The last condition makes it clear that $V$ is locally $\frakn_{(I)}$-torsion. Then the implication follows from implication $(4) \implies (1)$ of Lemma \ref{clm admissible 2}.

\subsection{Proofs of the Lemmas about preservation of admissibility}

\subsubsection{Proof of Lemma \ref{lem pind pres adm}}

	We first address $pres_I$. Using the definition and finiteness of the Harish-Chandra homomorphism $h_I : Z(\frakg) \to Z(\frakg_I)$, it is clear that $pres_I$ sends $Z(\frakg)$-finite modules to $Z(\frakg_I)$-finite modules. More precisely, one sees that given $z \in Z(\frakg)$, applying the functor $pres_I$ to the morphism $V \to V$ given by multiplication by $z$, one obtains the morphism $pres_I (V) \to pres_I (V)$ given by multiplication by $h_I (z)$. Also, since $\frakg = \frakn_{(I)} + \frakg_I + \frakk$, it is clear that $pres_I$ sends modules which are finitely generated over $U(\frakg)$ to modules which are finitely generated over $U(\frakg_I)$.
	
	\medskip
	
	We now address $pind_I$, solely exploiting it being the right adjoint of $pres_I$. Let $W$ be a $(\frakg_I , K_I)$-module having finite-dimensional isotypic components. We will show that $pind_I (W)$ also has finite-dimensional isotypic components. Let $E$ be a finite-dimensional $K$-module. Denote $V_E := U(\frakg) \otimes_{U(\frakk)} E$. Then, for a $(\frakg_I , K_I)$-module $W$, we have $$ \Hom_K (E , pind_I (W)) \cong \Hom_{(\frakg , K)} (V_E , pind_I (W)) \cong \Hom_{(\frakg_I , K_I)} (pres_I (V_E) , W).$$ Since $V_E$ is finitely generated over $U(\frakg)$, by what we have seen $pres_I (V_E)$ is finitely generated over $U(\frakg_I)$. As $W$ has finite-dimensional isotypic components, it is clear that the last $\Hom$-space is finite-dimensional, and thus so is the first, showing that $pind_I (W)$ has finite-dimensional isotypic components.
	
	\medskip
	
	Finally, let us show that if $W$ is a $Z(\frakg_I)$-finite $(\frakg_I , K_I)$-module, then $pind_I (W)$ is $Z(\frakg)$-finite. More precisely, we will show that given $z \in Z(\frakg)$, the morphism $pind_I (W) \to pind_I (W)$ given by multiplication by $z$ is equal to the morphism gotten by applying $pind_I$ to the morphism $W \to W$ given by multiplication by $h_I (z)$. For this, it is enough to show that for every $(\frakg , K)$-module $V$, two endomorphisms of $$ \Hom_{\frakg , K} (V , pind_I (W)),$$ the first obtained via the multiplication by $z$ on $pind_I (W)$, and the second obtained via the multiplication by $h_I (z)$ on $W$ - coincide. We can interpret the first endomorphism as given via the multiplication by $z$ on $V$, and identifying $$ \Hom_{\frakg , K} (V , pind_I (W)) \cong \Hom_{\frakg_I , K_I} (pres_I (V) , W),$$ we further interpret it, in view of what was said about $pres_I$ above, as given via the multiplication by $h_I (z)$ on $pres_I (V)$. On the other hand, the second endomorphism gets interpreted on the latter $\Hom$-space still as given via the multiplication by $h_I (z)$ on $W$. These interpretations show that our two endomorphisms indeed coincide.

\subsubsection{Proof of Lemma \ref{lem cofib delta fib adm}}

	That $cofib_I$ preserves admissibility is the contents of the implication $(1) \implies (2)$ of Lemma \ref{clm admissible 2}.
	
	\medskip
	
	Let $W$ be a $(\frakg_I , K_I)$ module. Since $\Delta_I (W)$ is generated over $U(\frakg)$ by a (twist of) $W$, it is clear that $\Delta_I (W)$ is finitely generated over $U(\frakg)$ if $W$ is finitely generated over $U(\frakg_I)$, and, using the Harish-Chandra homomorphism $Z(\frakg) \to Z(\frakg_I)$, that $\Delta_I (W)$ is $Z(\frakg)$-finite if $W$ is $Z(\frakg_I)$-finite.
	
	\medskip
	
	That $fib^-_I$ sends $Z(\frakg)$-finite modules to $Z(\frakg_I)$-finite modules is shown exactly as the corresponding claim for $pres_I$. That $fib^-_I$ sends modules which are finitely generated over $U(\frakg)$ to modules which are finitely generated over $U(\frakg_I)$ is again shown similarly to the corresponding claim for $pres_I$, where we use now $\frakg = \frakn_{(I)}^- + \frakg_I + \frakn_{(I)}$.

\subsubsection{Proof of Lemma \ref{lem B C adm}}

	Let us first notice that $\calB_I$ sends $Z(\frakg)$-finite modules to $Z(\frakg)$-finite modules. More precisely, one sees that given $z \in Z(\frakg)$, applying the functor $\calB_I$ to the morphism $W \to W$ given by multiplication by $z$, one obtains the morphism $\calB_I (W) \to \calB_I (W)$ given by multiplication by $z$. This is clear from the defining formula for $\calB_I$.
	
	\medskip
	
	To show that $\calB_I$ preserves admissibility, let us fix an ideal of finite codimension $J \subset Z(\frakg)$ and consider an admissible $(\frakg , K_I N_{(I)})$-module $W$ on which $J$ acts by zero. Depending only on $J$, there exists a finite set $S \subset (\fraka_{cent,I})^*_{\bbC}$ such that $wt_{\fraka_{cent,I}} (W^{\frakn_{(I)}}) \subset S$. We will prove the admissibility of $\calB_I (W)$ by induction on the number of elements in $wt_{\fraka_{cent,I}} (W) \cap S$ (if this number is zero, then $W = 0$ and the claim is clear). Notice that the counit map $\Delta_I (cofib_I (W)) \to W$ is an isomorphism on a generalized eigenspace $(\cdot)_{\fraka_{cent,I} , \lambda}$ whenever $\lambda \in wt_{\fraka_{cent,I}} (W) \cap S$ is maximal w.r.t. the partial order given by $\lambda_1 \leq \lambda_2$ if $\lambda_2 - \lambda_1 \in \sum_{\alpha \in wt_{\fraka_{cent,I}}(\frakn_{(I)})} \bbZ_{\ge 0} \cdot \alpha$. Therefore, considering the exact sequence $$ \Delta_I (cofib_I (W)) \to W \to C \to 0$$ (where $C$ simply denotes the cokernel of the counit map), and applying $\calB_I$ to it, we reduce ourselves to showing that $\calB_I (\Delta_I (cofib_I (W)))$ and $\calB_I (C)$ are admissible. But $\calB_I (\Delta_I (cofib_I (W))) \cong pind_I (cofib_I (W))$ (here we used Proposition \ref{prop rel B Delta pind}), which is admissible since $W$ is (as we have already shown that $pind_I$ and $cofib_I$ preserve admissibility), and we are thus reduced to showing that $\calB_I (C)$ is admissible. Notice that $C$ is again an admissible $(\frakg , K_I N_{(I)})$-module on which $J$ acts by zero, and that $wt_{\fraka_{cent,I}} (C) \cap S$ is contained properly in $wt_{\fraka_{cent,I}} (W) \cap S$, as it lacks the maximal elements. Therefore, by the induction hypothesis, $\calB_I (C)$ is admissible.
	
	\medskip
	
	We now address $\calC_I$, solely exploiting it being the right adjoint of $\calB_I$. We first show that $\calC_I$ sends $Z(\frakg)$-finite modules to $Z(\frakg)$-finite modules. More precisely, given a $(\frakg , K)$-module $V$ and $z \in \calZ (\frakg)$, the morphism $\calC_I (V) \to \calC_I (V)$ given by multiplication by $z$ is equal to the morphism gotten by applying $\calC_I$ to the morphism $V \to V$ given by multiplication by $z$. In fact, one deduces this from the corresponding fact for $\calB_I$ noticed above, in complete analogy with the parallel treatment for $pind_I$ in the last paragraph of the proof of Lemma \ref{lem pind pres adm}, so we skip this.
	
	\medskip
	
	Finally, we will show that given an admissible $(\frakg , K)$-module $V$, the $(\frakg , K_I N_{(I)})$-module $\calC_I (V)$ is also admissible. We just mentioned that $\calC_I (V)$ is $Z(\frakg)$-finite, therefore by Lemma \ref{clm admissible 2} it is enough to show that $cofib_I (\calC_I (V))$ is an admissible $(\frakg_I , K_I)$-module. Moreover, again since $\calC_I (V)$ is $Z(\frakg)$-finite, we already know that $cofib_I (\calC_I (V))$ is $Z(\frakg_I)$-finite (see the preliminary observations in the proof of \ref{clm admissible 2}), and it is therefore enough to see that $cofib_I (\calC_I (V))$ has finite-dimensional isotypic components. Let $E$ be a finite-dimensional $K_I$-module. Denote $W_E := U(\frakg_I) \otimes_{U(\frakk_I)} E$ (a $(\frakg_I , K_I)$-module) and denote by $J \subset Z(\frakg_I)$ an ideal of finite codimension which acts on $cofib_I (\calC_I (V))$ by zero. Then $W_E / J W_E$ is an admissible $(\frakg_I , K_I)$-module, and we have $$ \Hom_{K_I} (E , cofib_I (\calC_I (V))) \cong \Hom_{(\frakg_I , K_I)} (W_E / J W_E , cofib_I (\calC_I (V))) \cong $$ $$ \cong \Hom_{(\frakg , K)} (\calB_I (\Delta_I (W_E / J W_E)) , V) \cong \Hom_{(\frakg , K)} (pind_I (W_E / J W_E) , V).$$ Since both $W_E / J W_E$ and $V$ are admissible, the last $\Hom$-space is finite dimensional, and therefore so is the first, and hence the desired conclusion.

\subsection{Proof of Lemma \ref{lem auto cont}}

	Let us fix $$ \ell \in U^{* , \Delta \frakg_I + (\frakn_{(I)}^- \oplus \frakn_{(I)})}.$$ Using the Casselman-Wallach Theorem (Theorem \ref{Casselman-Wallach}) and standard Frobenius reciprocity, one has an identification $$ \Hom_{\frakg \oplus \frakg , K \times K} (U , C^{\infty}(Y_I)) \cong \Hom_{G(\bbR) \times G(\bbR)} (U^{\infty} , C^{\infty} (Y_I)) \cong (U^{\infty})^{\star , \Delta \frakg + (\frakn_{(I)}^- \oplus \frakn_{(I)})},$$ where $(\cdot)^{\star}$ denotes the space of continuous functionals. Therefore, we see that we simply need to show that $\ell$ extends to a continuous functional on $U^{\infty}$.

\medskip

We consider the parabolic subgroup $G_I N_{(I)}^- \times G_I N_{(I)}$ in $G \times G$ (defined over $\bbR$), and denote (just for this proof) by $$pres : \calM^a (\frakg \oplus \frakg , K \times K) \rightleftarrows \calM^a (\frakg_I \oplus \frakg_I , K_I \times K_I) : pind$$ the corresponding unnormalized parabolic restriction and induction functors. We want to see first that a continuous dashed arrow making the following diagram commutative, exists: $$ \xymatrix{ U^{\infty} \ar@{-->}[r] & pres (U)^{\infty} \\ U \ar[u] \ar[r] & pres(U) \ar[u]}.$$ One has the unit map $U \to pind (pres (U))$, and corresponding to it the map of representations $U^{\infty} \to pind (pres (U))^{\infty}$. It is well-known and not hard to establish, for $W \in \calM^a (\frakg_I \oplus \frakg_I , K_I \times K_I)$, an isomorphism $pind (W)^{\infty} \cong \textnormal{\textbf{pind}} (W^{\infty})$, where $\textnormal{\textbf{pind}} (\cdot)$ is the ``usual" parabolic induction construction, consisting of smooth functions on $G(\bbR) \times G (\bbR)$ which satisfy a transformation rule, etc. We clearly have a map $\textnormal{\textbf{pind}} (W^{\infty}) \to W^{\infty}$ given by evaulating at $1$, which gives us the composition $$ U^{\infty} \to pind(pres(U))^{\infty} \cong \textnormal{\textbf{pind}} (pres(U)^{\infty}) \to pres(U)^{\infty},$$ which is the desired arrow.

\medskip

The functional $\ell$ factors as the projection $U \to pres (U)$ followed by a functional $\ell^{\prime} \in pres(U)^{*,\Delta \frakg}$. We therefore see, using the commutative diagram above, that it is enough to show that $\ell^{\prime}$ extends to a continuous functional on $pres(U)^{\infty}$. This, in its turn, is a well-known ``automatic continuity" for symmetric subgroups (\cite[Th\'eor\`eme 1]{BaDe}).

\end{document}